\documentclass[pdftex,a4paper,12pt]{amsart}

\usepackage{geometry}
\usepackage[cp1250]{inputenc}

\usepackage{alltt}
\usepackage{euler}
\usepackage{amsfonts}
\usepackage{amscd}
\usepackage{graphicx}
\usepackage{lpic}
\usepackage{longtable}
\usepackage{url}
\usepackage{import}
\usepackage{amssymb, amsthm, amsmath}
\usepackage[sc]{mathpazo}
\usepackage{enumitem, textcomp, setspace}
\usepackage{nicefrac}
\usepackage{varwidth}

\usepackage{tikz}
\usepackage{tikz-cd}
\usetikzlibrary{bbox}
\usetikzlibrary{intersections}
\usepackage{pgfplots}
\usepgfplotslibrary{external}

\allowdisplaybreaks

\numberwithin{equation}{section}

\theoremstyle{plain}
\newtheorem{theorem}{Theorem}[section]
\newtheorem{proposition}[theorem]{Proposition}

\theoremstyle{definition}

\newtheorem{example}[theorem]{Example}

\usepackage[colorlinks, pagebackref=false]{hyperref}

\usepackage[all]{hypcap}

\definecolor{internalLink}{rgb}{1,0,0.5}
\definecolor{citeLink}{rgb}{0,0.5,0}
\definecolor{urlLink}{rgb}{0,0.5,0.5}

\hypersetup{
	linkcolor=internalLink,
	citecolor=citeLink,
	urlcolor=urlLink
}

\title[CWR sequence of invariants of alternating links and its properties]{CWR sequence of invariants of alternating links\\ and its properties}

\author{Micha{\l} Jab{\l}onowski}

\address{Institute of Mathematics, Faculty of Mathematics, Physics and Informatics,\newline University of Gda\'nsk, 80-308 Gda\'nsk, Poland}



\email{michal.jablonowski@gmail.com}

\date{\today}

\begin{document}

\maketitle

\begin{abstract}
	
We present the $CWR$ invariant, a new invariant for alternating links, which builds upon and generalizes the $WRP$ invariant. The $CWR$ invariant is an array of two-variable polynomials that provides a stronger invariant compared to the $WRP$ invariant. We compare the strength of our invariant with the classical HOMFLYPT, Kauffman $3$-variable, and Kauffman $2$-variable polynomials on specific knot examples. Additionally, we derive general recursive "skein" relations, and also specific formulas for the initial components of the $CWR$ invariant using weighted adjacency matrices of modified Tait graphs.

\end{abstract}
\vspace{-0.3cm}
\section{Introduction}

The paper concerns embeddings of a circle in three-dimensional space, known as a knot (or link if there can be more than one component). A significant aspect of knot theory is the development and application of knot invariants, mathematical constructs that distinguish different knots. Among these invariants, polynomial invariants have played a pivotal role. Notable examples include the Alexander polynomial, the Jones polynomial, and the HOMFLYPT polynomial. These invariants have been instrumental in classifying knots and links, particularly alternating links, which are links with diagrams where over-crossings and under-crossings alternate as one travels along each component.
\par 
In this paper, we introduce a new invariant for alternating links called the $CWR$ invariant. This invariant is a sequence of two-variable polynomials that generalizes the previously defined $WRP$ invariant. The $WRP$ invariant, while useful, has limitations in its ability to distinguish between certain knots. Our new invariant overcomes these limitations by incorporating additional structural information from the knot’s diagram.
\par 
In this paper, Section\;\ref{s1} contains necessary definitions and an example of $CWR$ calculations, and we give proof of $CWR$ invariance of flype-move, needed to apply the proven Tait's Flype Conjecture. In Section\;\ref{s2} we show that $CWR$ can distinguish between knots that other well-known invariants, such as HOMFLYPT and Kauffman polynomials, cannot. We then give properties of the invariant concerning: connected sum operation, chirality, and mutation. Additionally, we show that the $CWR2$ and $CWR3$ invariants can be computed using weighted adjacency matrices of modified Tait graphs, providing a practical computational method. Section\;\ref{s3} contains a recursive relation for $CWR$ for diagrams satisfying given conditions. In Section\;\ref{s4} of this paper, we show computationally generated tables of the invariant for prime knots up to eight crossings.

\section{Definitions}\label{s1}

A \emph{diagram} $D$ of a (non-trivial) knot or link $K$ is a generic projection of the knot or link in the plane, a regular $4$-valent graph, with extra information at each vertex indicating which arc of the link passes over the other. Moreover it is \emph{alternating} when traveling around each link component on a diagram we pass crossing in alternate type fashion (i.e. if one passes an over/under-crossing, the next crossing will be an under/over-crossing, see \cite{Men21} for a recent survey).
\par 
A knot or link diagram is a \emph{reduced diagram} if a diagram without \emph{nugatory crossings}, i.e a crossing, such that its neighborhood looks like in Figure \ref{rys17_1} left, that can be removed without changing the knot type,  and it cannot be separated by any circle in the plane of projection (i.e. the diagram as projection has one connected component).

\begin{figure}[h!t]
	\begin{center}
		\includegraphics[width=10cm]{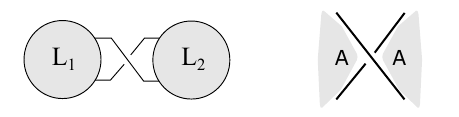}
		\caption{A nugatory crossing (left) and a checkerboard coloring (right).\label{rys17_1}}
	\end{center}
\end{figure}

By a \emph{region} of a reduced diagram $D$ we mean a face of the underlying projection of the link. The diagram can be checkerboard-shaded so that every edge separates a shaded region (black region) from an unshaded one (white region).
\par 
Moreover, we can assume \cite{Lis47, Men21} that if the diagram is reduced then any region of type $A$ as shown in Figure \ref{rys17_1} right is a black region, otherwise it is a white region. Therefore, from now on, we can always assume this convention of coloring of diagram regions, when the diagram is reduced and alternating.
\par   
We consider two graphs associated with $D$ (also known as Tait graphs) the shaded checkerboard graph $G_B$ and the unshaded checkerboard graph $G_W$. The graph $G_B$ has vertices corresponding to the black region regions of $D$, and an edge between two vertices for every crossing at which the corresponding regions meet. The dual graph $G_W$ is defined by analogy but the vertices correspond to the white regions.
\par
On each edge $e$ of graph $G_B$ and $G_W$ we assign its weight $wg(e)$ as a variable $w$ when the corresponding crossing was positive and variable $r$ if it was negative crossing (see Figure \ref{rys17_6}).

\begin{figure}[h!t]
	\begin{center}
		\includegraphics[width=8cm]{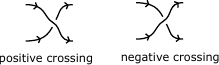}
		\caption{The convention for positive and negative crossings.\label{rys17_6}}
	\end{center}
\end{figure}

\par 
An unoriented graph $G_B^*$ is then obtained from an unoriented graph $G_B$ by consolidating all multiple edges between the same pair of vertices to a single edge that has its weight equal to the product of weights all edges being consolidated. The same goes with defining graph $G_W^*$ from $G_W$.

\par 
For $i>2$, we define $CB_i(w,r)\in\mathbb{Z}[w,r]$ as a two-variable polynomial $\sum_{C_B}\prod_{e}{wg(e)}$ where the sum is taken over all unoriented cycles (i.e. simple closed paths) $C_B$, of length $i$, in a graph $G_B^*$ of the diagram $D$ and the product is taken over all edges in a given cycle. \\ We define $CB_2(w,r)\in\mathbb{Z}[w,r]$ as a sum $\sum_{e}{wg(e)}$ where the sum is taken over all edges in $G_B^*$ (when there are no edges in case of the trivial knot we can put the value as $0$).
\par 
By analogy we define $CW_i(w,r)\in\mathbb{Z}[w,r]$ for all $i>1$ considering a graph $G_W^*$\\ instead of $G_B^*$. For an integer $i>1$, we define $CWR_i(w,r)$ as an ordered tuple\\ $(CB_i(w,r), CW_i(w,r))$, when there are no cycles of length $i>2$ in a given graph we put the corresponding value as $0$. Define $CWR$ of an alternating knot or link $L$ as a sequence of invariants (ordered tuple): $$CWR(L)=(CWR_2, CWR_3, CWR_4, \ldots)$$ of any reduced alternating diagram $D_L$ of $L$. When $CWR_i(w,r)=(0,0)$ for every $i>N$, for some integer $N>2$ then we omit those values $CWR_i(w,r)$ in the sequence for $CWR$, therefore the sequence is finite as there are finitely many cycles in a finite graph.

\subsection{An example}

For the knot $K7a1$ we have calculations show in Figure \ref{K7a1}, that is $CWR(K7a1)= ((4r + 3w, 2r^2 + 3w), (2r^2w + w^3, 2r^2w^2), (2r^2w^2, r^4w^2), (r^4w, 0))$.

\begin{figure}[h!t]
	\begin{center}
		\includegraphics[width=10cm]{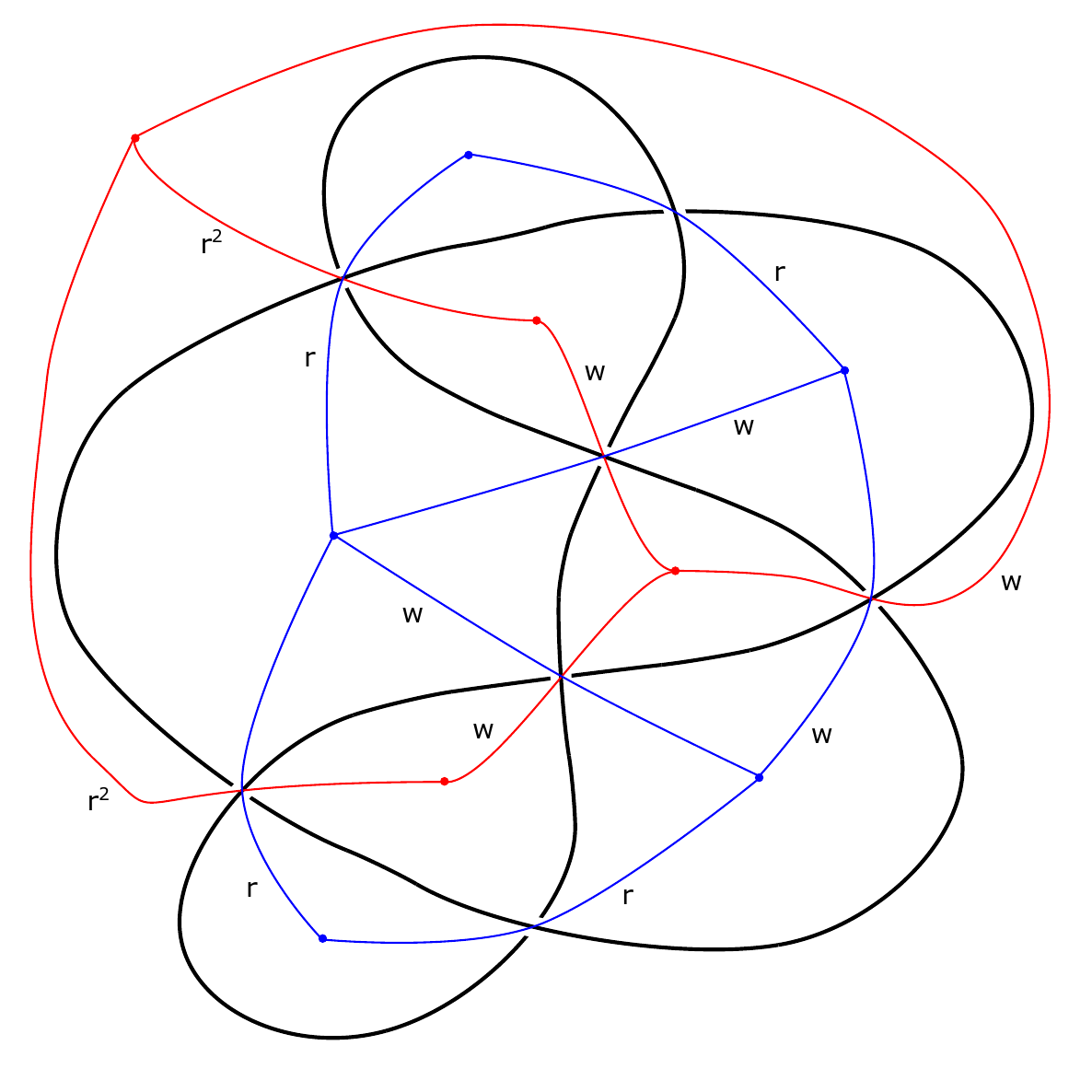}
		\caption{Knot $K7a1$ with their graphs $G_B^*$ (blue) and $G_W^*$ (red).\label{K7a1}}
	\end{center}
\end{figure}

\subsection{An invariance}

\begin{theorem}\label{twA}
	For any pair $D_1$ and $D_2$ of reduced alternating diagrams of a given alternating link, we have $$CWR(D_1)=CWR(D_2).$$
\end{theorem}

\begin{proof}
	
	We use the theorem known as Tait's Flype Conjecture, which states that if $D_1$ and $D_2$ are reduced alternating diagrams, then $D_2$ can be obtained from $D_1$ by a sequence of flype-moves (a local move on a knot diagram that involves flipping a part of the diagram, see Figure \ref{rys17_2}) if and only if $D_1$ and $D_2$ represent the same link. This conjecture has been proven true by Menasco and Thistlethwaite \cite{MenThi91, MenThi93}.
	
		\begin{figure}[h!t]
		\begin{center}
			\includegraphics[width=9cm]{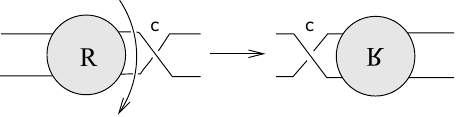}
			\caption{A flype-move on the tangle $R$ moves the crossing $c$ from one side of $R$ to the other, while rotating $R$ by $180^\circ$ around the horizontal axis.\label{rys17_2}}
		\end{center}
	\end{figure}
	
	\begin{figure}[h!t]
		\begin{center}
			\includegraphics[width=13cm]{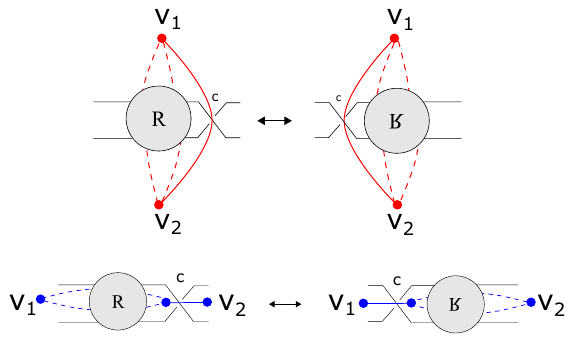}
			\caption{The general situation on graphs in flype-moves. Case I and II.\label{rys17_3}}
		\end{center}
	\end{figure}
	
	The diagrams can be checkerboard-colored, resulting in two types of regions (black and white). Let us fix the reduced alternating diagram before the flype as $D_1$ and after the flype as $D_2$. We can now fix the graph $G_1$ as one of $G_B^*$ or $G_W^*$ for a diagram $D_1$ and the graph $G_2$ as the corresponding graph for a diagram $D_2$. First, we notice that the black and white corresponding regions before and after the flype-move preserves their color. Second, we notice that the flype-move preserves the length of any corresponding cycles in graphs $G_B^*$ and $G_W^*$, and also preserves the number of consolidated edges crossings and their weights.
	\par 
	Therefore, for $i>1$ in $CB_i(w,r)$ (and $CW_i(w,r)$ respectively) there is a one-to-one correspondence between cycle summands in the definition of the functions. The rest of the proof is in the spirit of the proof of invariants of our previous (weaker) invariant $WRP$ (see \cite{Jab24}). Namely, it is sufficient to prove that the corresponding cycles preserve the product of the weights of all their edges (that forms the cycle).
	\par
 	This can be seen in the general situation in Figure \ref{rys17_3} for the cases of black or white graphs $G_B^*$ and $G_W^*$ we are considering, that only the relative position of edges can vary between corresponding cycles so the product is preserved (as a commutative operation).  We can easily conclude that if any cycle from $G_1$ does not go through both $v_1$ and $v_2$ then there is a one-to-one correspondence between cycles in $G_1$ and cycles in $G_2$ and each corresponding cycle preserves its number of edges and their weights. It is because the cycles are either outside the shown flype tangle (so they do not change) or they are completely inside a tangle $R$ (plus eventually touching the unlabeled dot in the figure).
	\par 
	When a cycle in $G_1$ passes through both $v_1$ and $v_2$ then there is a corresponding cycle in $G_2$ that in Case I: preserves its number, its order, and their weight of edges involved in the cycle, in Case II: for all edges between $v_1$ and $v_2$, the weights are preserved but the order may vary so this does not change the product of weights in the cycle.

\end{proof}

\section{Properties}\label{s2}

Let us consider the classical alternating knots and non-split link polynomial invariants: HOMFLYPT \cite{HOMFLY85, PrzTra88}, Kauffman $3$-variable \cite{Kau89, Laf13} and Kauffman $2$-variable \cite{Kau90} polynomials, and call them here: HOMFLYPT, Kauffman3v and Kauffman2v respectively.
\par 

\noindent
	\begin{center}
		\begin{figure}[h!b]
\begin{tikzcd}[column sep=1.75cm, row sep=4cm]
	&&\text{CWR}\arrow[bend left=10,pos=.7]{d}{12a24, 12a299}\arrow[bend left=30 ,pos=.4]{ddrr}{11a67,11a317}\arrow[bend left=-15,pos=.4]{ddll}{11a30, 11a189m}&&\\
	&&\text{HOMFLYPT}\arrow[bend left=10]{u}{?}\arrow[bend left=-1,pos=.2]{drr}{12a171, 12a217m}\arrow[bend left=17,pos=.6]{dll}{11a189, 11a30m}&&\\
	\text{Kauffman2v}\arrow[bend left=5,swap]{rrrr}{13a227, 13a1016m}\arrow[bend left=30]{uurr}{?}\arrow[bend left=-1,pos=.8]{urr}{11a1,11a149}&&&&\text{Kauffman3v}\arrow[bend left=10]{llll}{13a219, 13a2738 }\arrow[bend left=-15]{uull}{?}\arrow[bend left=17,pos=.6]{ull}{12a5, 12a343}\\
\end{tikzcd}
\vspace{-3cm}
	\caption{Invariants and relations}
\label{g1}
\end{figure}
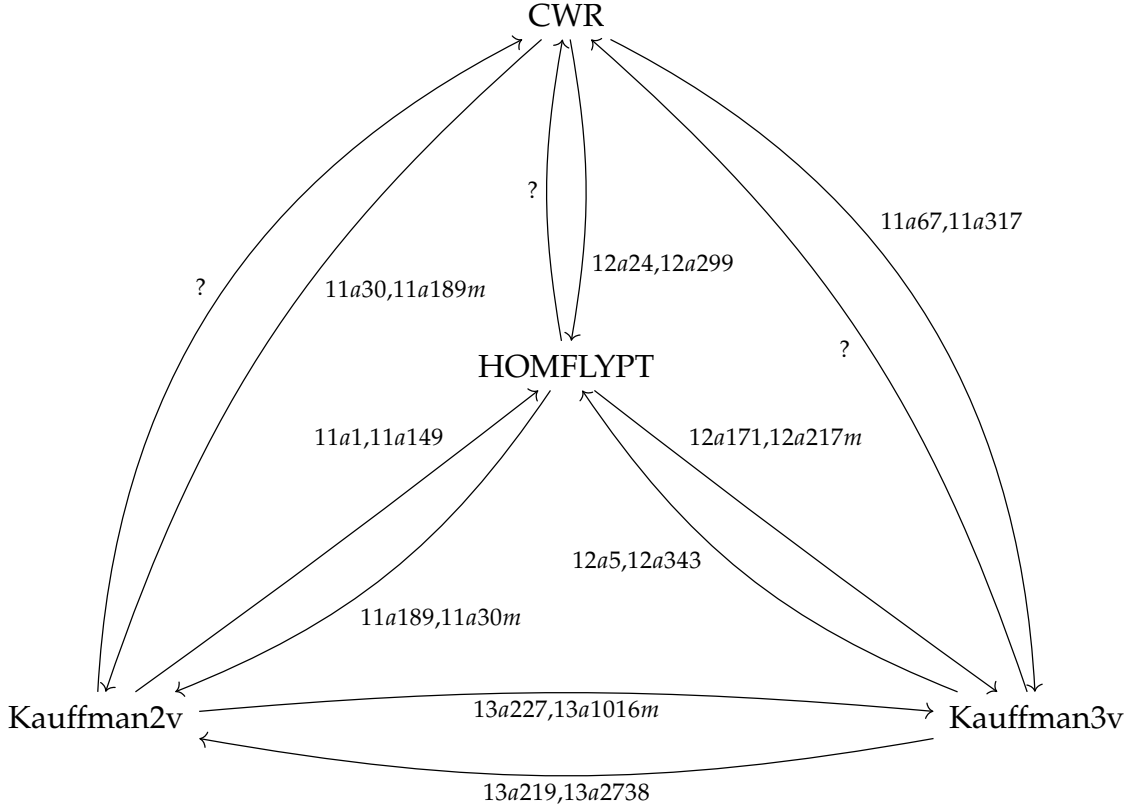
	\end{center}
	
Figure \ref{g1} shows the HOMFLYPT, Kauffman3v, and Kauffman2v invariant examples, illustrating that these invariants are generally non-comparable. An arrow $A\rightarrow B$ means that the invariant $A$ distinguishes the knots indicated in the arrow label and $B$ does not. The letter $m$ after a knot name stands for the mirror image of that knot. There are examples (shown in Figure \ref{g1}) where $CWR$ invariant is stronger than the mentioned other three invariants. But we cannot find examples when $CWR$ is strictly weaker, we have computationally searched in the knot tables of knots \underline{up to $14$ crossings}, including their mirror images. Examples in the labels were found using \cite{KAtlas, SnapPy, SageMath}.
\par
In addition, we find that for a knotted pair $K12a24$ and $K12a299$ are indistinguishable by HOMFLYPT, Kauffman3v, Kauffman2v, signature, unknotting number, bridge index, $4$-genus, Khovanov and Knot Floer homologies \cite{KAtlas, SnapPy, SageMath} (see \cite{LivMoo23} for a few other invariants that do not tell these knots apart). To distinguish we can use our $CWR$ invariant, we have:

$CWR(K12a24) = ((4r + 8w, w^3 + 2r^2 + w^2 + 3w), (r^2w, r^2w^4 + r^2w^3 + r^2w^2),\\ (2w^4, r^4w^3 + w^7), (3r^2w^3, r^2w^7), (r^4w^2, 0), (2r^2w^5, 0), (r^4w^4, 0))$, and

$CWR(K12a299) = ((4r + 8w, w^3 + 2r^2 + w^2 + 3w), (w^3, r^2w^4 + r^2w^3 + r^2w^2),\\ (2r^2w^2, r^4w^4 + w^5), (2r^2w^3 + w^5, r^2w^7), (r^4w^2, 0), (r^4w^3 + r^2w^5, 0), (r^2w^6, 0)).$

\begin{figure}[h!t]
	\begin{center}
		\includegraphics[width=6cm]{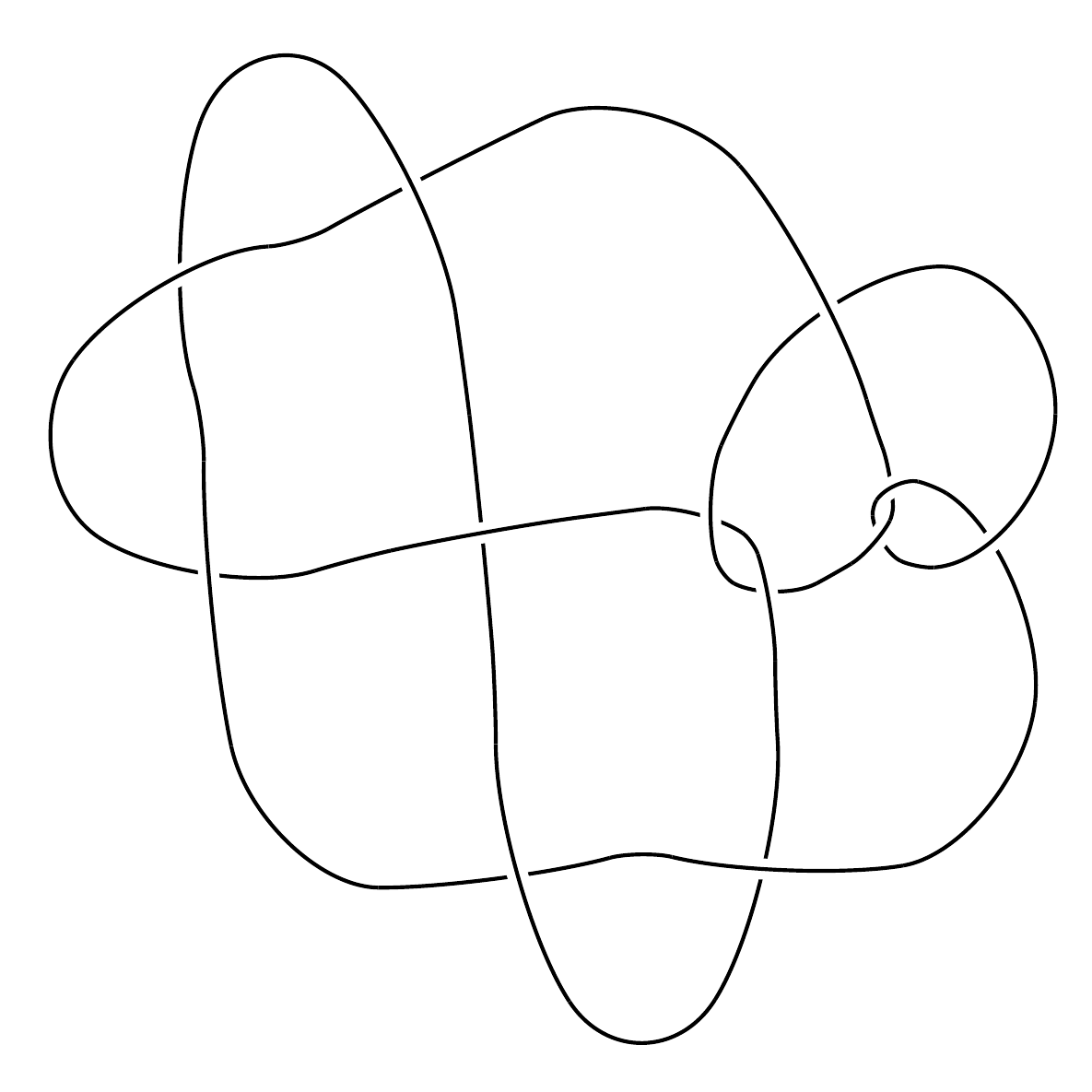}
		\includegraphics[width=6cm]{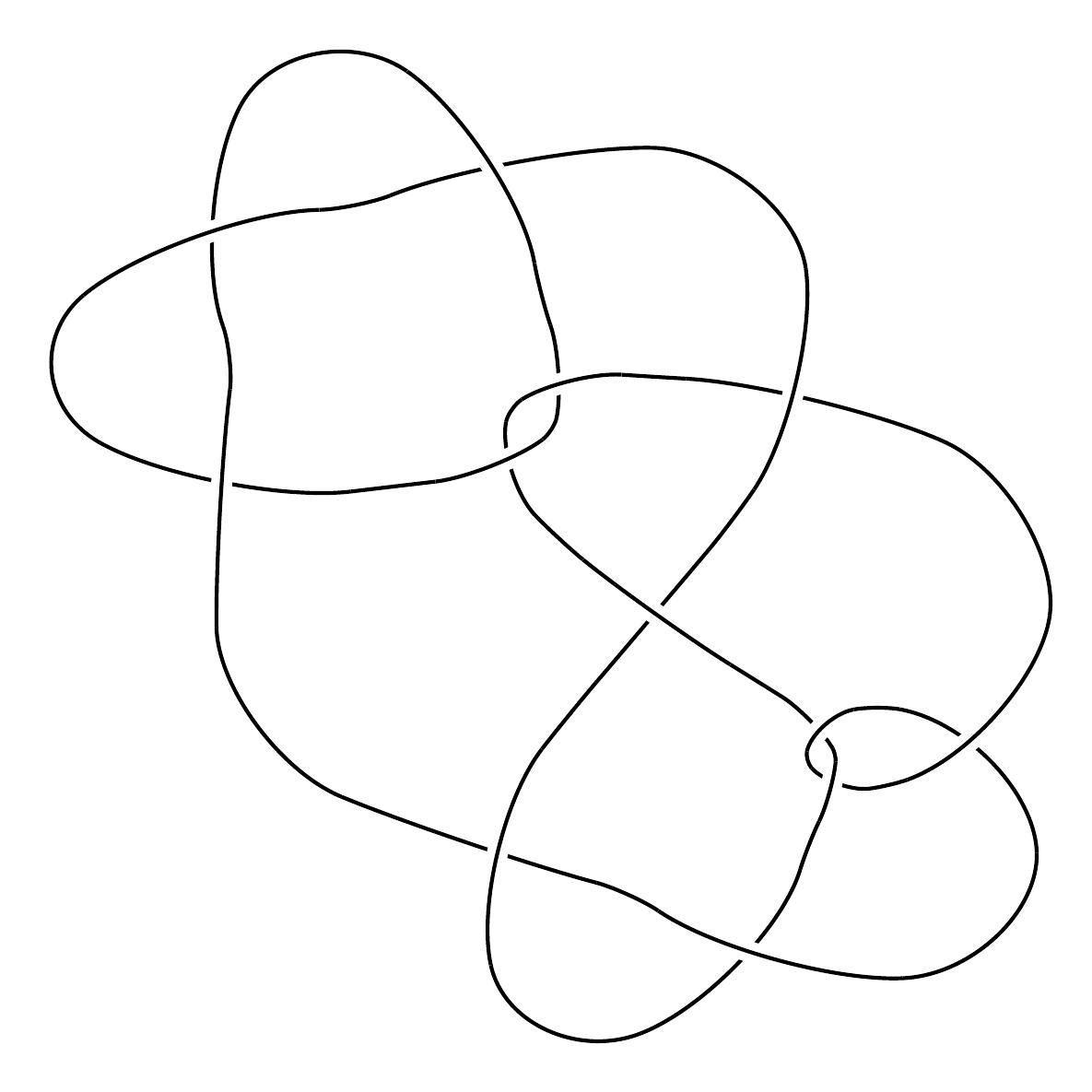}
		\caption{Knots $K12a24$ and $K12a299$.\label{rys2}}
	\end{center}
\end{figure}

\subsection{Generalization of the other invariants}

With our definition, $CWR$ invariant generalize the $WRP$ invariant defined in one of our previous papers \cite{Jab24} as follows:
$$WRP(w,r)=\left\{CWR_2(w^2,r^2)+2\cdot\sum_{i>2} CWR_i(w,r)\right\},$$
where the curly brackets mean taking the ordered tuple to an unordered tuple, and the sum of tuples is performed as the sum of corresponding coordinates in the tuples.
\par 
$CWR$ invariant is strictly stronger than $WRP$, for example, the pair $K11a75$ and $K11a102$ cannot be distinguished by $WRP$ as shown in \cite{Jab24} (together with their diagrams and resulting graphs), but $CWR$ tell them apart, because:

$CWR(K11a75) = ((w^3 + 4r + 4w, w^3 + 2r^2 + 4w),
(r^2w^3 + r^2w, r^2w^4), (0, 0),\\
(w^7, r^2w^6),
(r^2w^6 + r^2w^4, r^4w^4),
(r^4w^3, 0))$, and

$CWR(K11a102) = ((w^3 + 4r + 4w, w^3 + 2r^2 + 4w),
(r^2w, 0),
(r^2w^4, r^4w^4),\\
(r^4w^3 + w^7 + r^2w^3, r^2w^6 + r^2w^4),
(r^2w^6, 0)).$
\par 

From $CWR(L)$ invariant, we can naturally read the crossing number $cr(L)$ and $writhe(L)$, we have the following.

\begin{proposition} For any alternating non-split link $L$:
$$cr(L) = \left.\frac{\partial CB_2(L)}{\partial w}\right|_{w=1} +\left.\frac{\partial CB_2(L)}{\partial r}\right|_{r=1} = \left.\frac{\partial CW_2(L)}{\partial w}\right|_{w=1} +\left.\frac{\partial CW_2(L)}{\partial r}\right|_{r=1},$$

$$writhe(L) = \left.\frac{\partial CB_2(L)}{\partial w}\right|_{w=1} -\left.\frac{\partial CB_2(L)}{\partial r}\right|_{r=1} = \left.\frac{\partial CW_2(L)}{\partial w}\right|_{w=1} -\left.\frac{\partial CW_2(L)}{\partial r}\right|_{r=1},$$

where the $writhe(L)$ is the total number of positive crossings minus the total number of negative crossings in a reduced alternating diagram of $L$, and $cr(L)$ is the total number of crossings in a reduced alternating diagram of $L$.
\end{proposition} 

\begin{proof}
	The alternating knot (or link) invariants $cr(L)$, $writhe(L)$, $CB_i$ and $CW_i$ can be calculated from the link $L$ reduced alternating diagram (\cite{Kau88, Mur87, Thi87, Thi88}). In that kind of diagram, we can simply count each crossing as the sum of powers in the formula for $CB_2$ or $CW_2$ expressed as a sum of unimodal monomials, where each monomial is a one-variable function, because otherwise it would not be reduced alternating diagram as it can be reduced by a Reidemeister II move after possibly a flype-move.
\end{proof}

It is also useful in increasing efficiency when searching for different knots with the same $CWR$ invariant as they must have the same crossing number (and there are finitely many such knots).
\par
When counting the product of weights for the cycle contribution in the formula for $CWR(L)$, the sum of powers of $w$ and $r$ in each monomial cannot be greater than $cr(L)$ as this is the total number of edges in a considered graph (counting with multiplications for the consolidate edges).

\subsection{Connected sum}

$CWR$ is additive concerning the connected sum, we have the following.

\begin{theorem}
	If $K_1$ and $K_2$ are alternating knots, or non-split links, then
	$$CWR(K_1\#K_2)=CWR(K_1)+CWR(K_2),$$
	where the sum of tuples is performed as the sum of corresponding coordinates in the tuples.
\end{theorem}

\begin{proof}
Let $G_1\vee G_2$ denote a graph obtained from graphs $G_1$ and $G_2$ by gluing them along one vertex (identification of one arbitrarily chosen vertex from each graph). This operation in this context of graphs is also called block sum. It can be straightforwardly examined that the connected sum operation $K_1\#K_2$ can have the reduced alternating projection such that it merges exactly one black region and one white region from each checkerboard coloring of complementary regions of reduced alternating diagrams for $K_1$ and $K_2$. From this operation we get that $G_B^*(K_1\#K_2) = G_B^*(K_1)\vee G_B^*(K_2)$ and $G_W^*(K_1\#K_2) = G_W^*(K_1)\vee G_W^*(K_2)$.  Moreover, in the resulting graphs, no two edges merge and no edge is canceled. There are no new cycles besides the ones completely embedded in graphs for $G_1$ and $G_2$ because each new closed path passing through the joining vertex must pass it twice (otherwise it is one of the old cycles) and that contradicts the assumption that the path is simple. Therefore the cycles as a set just adds in counting $CWR(K_1\#K_2)$, causing $CWR(K_1\#K_2)=CWR(K_1)+CWR(K_2).$
\end{proof}

\subsection{Non-mutant knots}

\emph{Mutant knots} are the knots such that they have diagrams that differ only by a (Conway's) mutation operation. The mutation operation modifies a knot by first cutting, then rotating (by $\pi$) and finally gluing arcs to preserve the number of components (that is one component in our knot case), the operation is performed in a $2$-tangle along a $2$–sphere embedded in $\mathbb{R}^3$ and intersecting transversely the knot in exactly four points.

\begin{proposition}
	If $K_1$ and $K_2$ a mutant pair, then $CWR(K_1) = CWR(K_2)$.
\end{proposition}

\begin{proof}

Let $D_1$ and $D_2$ be two reduced alternating knot diagrams that differ by a single mutation. There is a natural bijection between the crossings of $D_1$ and $D_2$ that induces a bijection on the edge sets of the Tait graphs of $D_1$ and $D_2$. In fact, the signed Tait graphs of $D_1$ and $D_2$ are $2$-isomorphic (see \cite{ChaKof08, Che13}).
\par
The diagrams can be checkerboard-colored, resulting in two types of regions (black and white). We can fix the graph $G_1$ as one of $G_B^*$ or $G_W^*$ for a diagram $D_1$ and the graph $G_2$ as the corresponding graph for a diagram $D_2$. First, we notice that the black and white corresponding regions before and after the mutation preserve their color. Second, we notice that the mutation preserves the length of any corresponding cycles in graphs $G_B^*$ and $G_W^*$, and also preserves the number of consolidated edges crossings and their weights.
\par 
Therefore, there is a one-to-one correspondence between cycle summands in the definition of $CWR$ invariant. It is sufficient to prove that the corresponding cycles preserve the product of the weights of all their edges (that forms the cycle). That follows directly from the corresponding part of the proof of Theorem \label{twA} but without a crossing $c$ indicated there.
	
\end{proof}

Our computations show that the $CWR$ invariant is strong enough to distinguish all non-mutant knots in the tables with up to $11$ crossings, including their mirror images (when the knot is chiral).
An example of a non-mutant pair (see \cite{Sto24}) that the $CWR$ invariant does not distinguish is $K12a29$ and $K12a113m$, shown in Figure \ref{rys3}. The $CWR$ invariant in their case is equal to
$((2r^2 + 4r + 4w, r^2 + w^2 + 6r + 2w),
(r^2w, r^2w^2 + 2r^3),\\
(r^6 + 4r^3w^2 + w^4, r^2w^3),
(2r^5w + r^2w^3, r^4w^3 + 2r^3w^3),
(0, 2r^5w^3 + r^4w^3),
(0, r^6w^3)).$

\begin{figure}[h!t]
	\begin{center}
		\includegraphics[width=6cm]{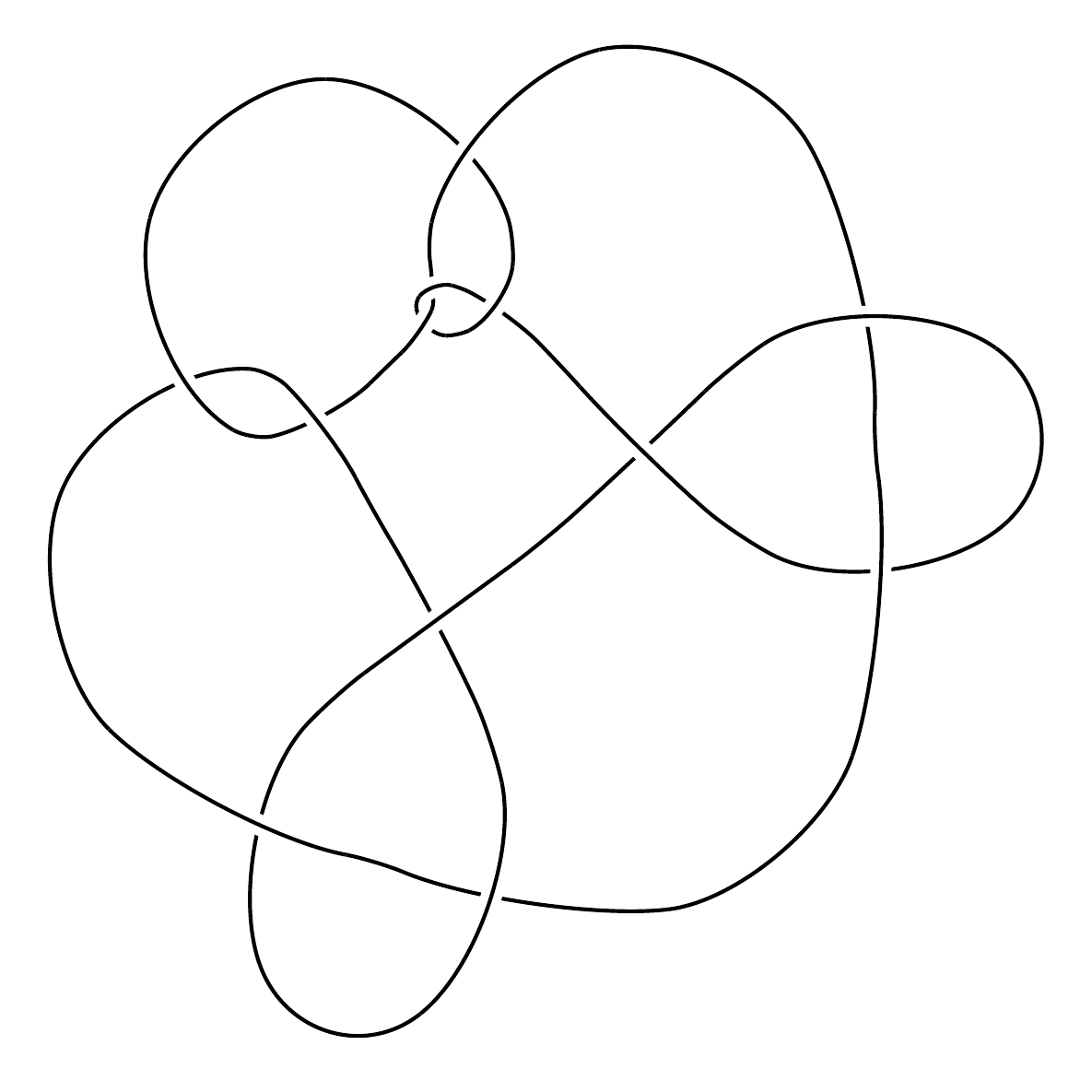}
		\includegraphics[width=6cm]{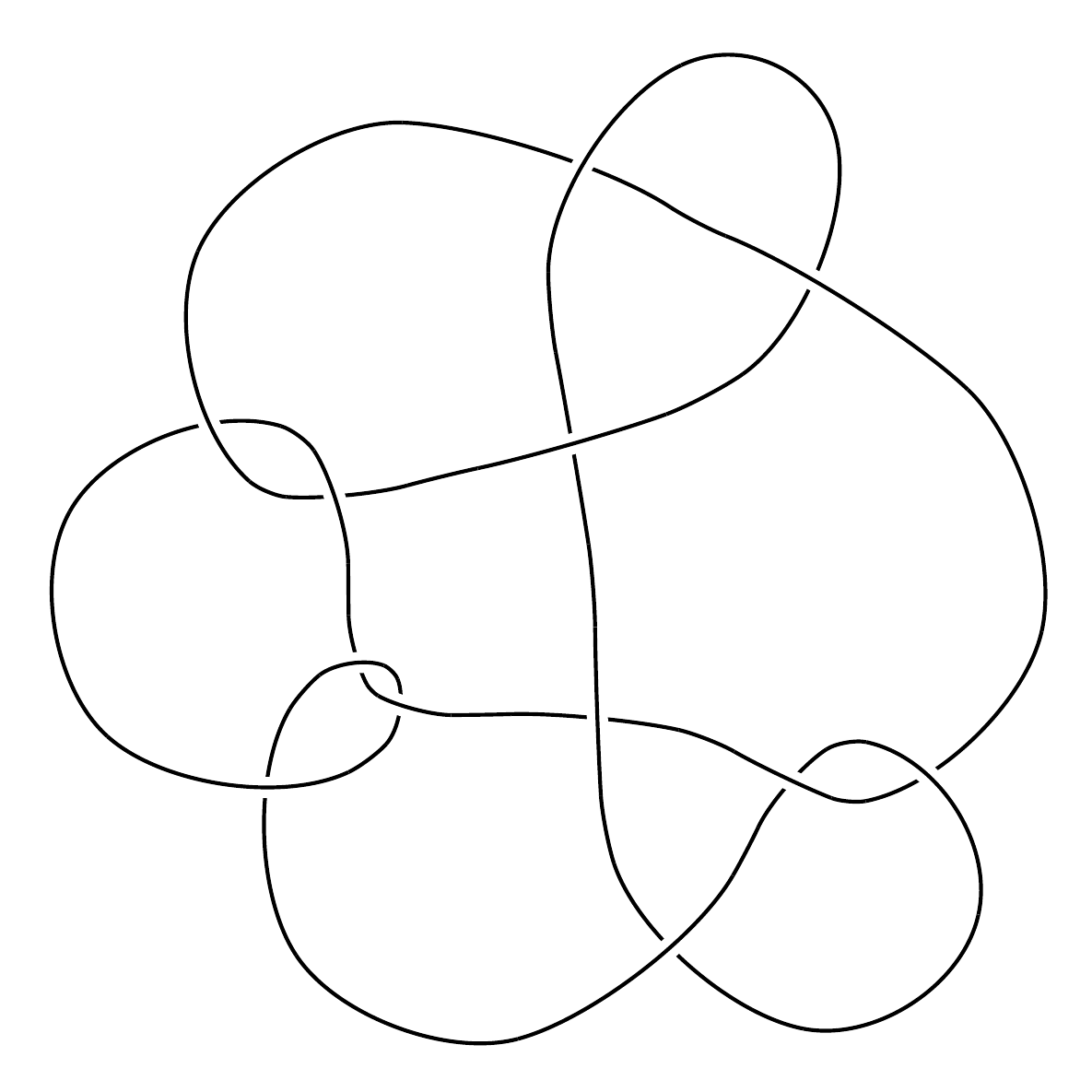}
		\caption{Knots $K12a29$ and $K12a113m$.\label{rys3}}
	\end{center}
\end{figure}

\subsection{Chirality}

We have the straightforward formula for the $CWR$ invariant for the mirror image of a given link knowing the $CWR$ value for the original link.
\begin{proposition}
	For any alternating link $L$, if we express $CWR_i$ as \\ $CWR_i(w,r)(L)=(CB_i(w,r),CW_i(w,r))$, then we have\\
	$CWR_i(w,r)(mirror(L))=(CW_i(r,w),CB_i(r,w))$, where $mirror(L)$ is the mirror\\ image of $L$.
\end{proposition}

\begin{proof}
	Let $D_m$ be a reduced alternating diagram for $mirror(L)$ obtained by switching crossing types of every crossing in a given reduced alternating diagram $D$ for $L$. Therefore, the diagrams $D$ and $D_m$ have corresponding regions in the projection complement but with switched colors, therefore the black and white graphs are exchanged. All cycles geometrically stay unchanged and the weights of each edge in a cycle are also switched, replacing variable $r$ with $w$ and vice versa because switching crossing types changes the sign of a crossing to the opposite. 
\end{proof}

In the set of alternating prime knots up to $13$ crossings, the $CWR(L)$ can tell apart a knot and its mirror image if it is chiral (see \cite{Sto24}). This is in contrast to invariants HOMFLYPT, Kauffman3v, and Kauffman2v which cannot distinguish between the knot $K10a10$ from its mirror image.
\par
The knot $K14a506$ is not equivalent to its mirror image, but\\ $CWR(K14a506)=CWR(K14a506m)$.

\subsection{Matrix formulae for $CWR_2$ and $CWR_3$}

The \emph{weighted adjacency matrix} of a (simple) graph $G$ without loops, is the symmetric matrix $\overline{A}(G)$ whose rows and columns are indexed by some consistent orderings of vertices $v_i$ such that $\overline{A}_{i,j}(G)=weight(e)$ if vertices $v_i$ and $v_j$ are endpoints of an edge $e$ in $G$ and $0$ otherwise. Define also the \emph{adjacency matrix} $A$ as a weighted adjacency matrix where all edge weights are equal to $1$.

\begin{proposition}
	Let $\overline{A}_B$ (resp. $\overline{A}_W$) be the weighted adjacency matrices for a graph $G_B^*$ (resp. $G_W^*$) of an alternating link $L$, and $A_B$ (resp. $A_W$) their corresponding adjacency matrices respectively.
	
	$$CWR_2(L)=\left(\text{trace}(\overline{A}_B\cdot A_B)/2,\;\;\text{trace}(\overline{A}_W\cdot A_W)/2\right),$$
	$$CWR_3(L)=\left(\text{trace}(\overline{A}_B^3)/6,\;\;\text{trace}(\overline{A}_W^3)/6\right).$$
\end{proposition}

\begin{proof}
	It follows from a standard Graph Theory argument, see for example \cite[Theorem 2.5.6]{KnaKol19}. Counting the number of closed directed paths of length $2$ and $3$ in a graph $G$ equal $\text{trace}(A^2)/2$ and $\text{trace}(A^3)/3$ respectively for a standard adjacency matrix $A$ of a graph $G$. We need only to consider weights. In case $CWR_2$ we count paths of length $2$ from each vertex to itself, so counting one edge with weight from $\overline{A}$ and one with weight $1$ back to be consistent with the definition of multiplication of this to weights equal just single weight. In the case of $CWR_3$, we count only simple paths (cycles) of length $3$ that are multiplied when forming a cycle invariant and added when summing all cycles in our invariant $CWR_3$, which is consistent with matrix multiplications in this context. The denominator here is $6$ because we consider unoriented graphs so the closed path in one direction is the same as the path in the opposite direction.
\end{proof}

\begin{example}
Let us calculate with the above method $CWR_2(K)$ and $CWR_3(K)$ invariants for knot $K=K7a1$ from its diagram shown in Figure \ref{K7a1}.

We have for the graphs $G_B^*$ and $G_W^*$ in this figure that $\overline{A}_B=\begin{bmatrix}
	0&0&r&r&0\\
	0&0&0&r&r\\
	r&0&0&w&w\\
	r&r&w&0&w\\
	0&r&w&w&0
	
\end{bmatrix}$,  \\ \; \\
$\overline{A}_W=\begin{bmatrix}
	0&r^2&r^2&w\\
	r^2&0&0&w\\
	r^2&0&0&w\\
	w&w&w&0
	
\end{bmatrix}$,
$\overline{A}_B\cdot A_B=\begin{bmatrix}
2r& r& r& r& 2r\\
r& 2r& 2r& r& r\\
w& 2w& r + 2w& r + w& w\\
w& w& r + w& 2r + 2w& r + w\\
2w& w& w& r + w& r + 2w
\end{bmatrix}$, \\ \; \\

$\overline{A}_W\cdot A_W=\begin{bmatrix}

	2r^2 + w& w& w& 2r^2\\
	w& r^2 + w& r^2 + w& r^2\\
	w& r^2 + w& r^2 + w& r^2\\
	2w& w& w& 3w
\end{bmatrix}$,\\ \; \\

$\overline{A}_B^3=\begin{bmatrix}
        2r^2w      &   3r^2w &2r^3 + 3rw^2& 3r^3 + 3rw^2&   r^3 + 2rw^2\\        3r^2w     &    2r^2w&   r^3 + 2rw^2 &3r^3 + 3rw^2& 2r^3 + 3rw^2\\2r^3 + 3rw^2  & r^3 + 2rw^2 &2r^2w + 2w^3 &4r^2w + 3w^3& 4r^2w + 3w^3\\3r^3 + 3rw^2& 3r^3 + 3rw^2 &4r^2w + 3w^3& 4r^2w + 2w^3& 4r^2w + 3w^3\\  r^3 + 2rw^2& 2r^3 + 3rw^2& 4r^2w + 3w^3& 4r^2w + 3w^3& 2r^2w + 2w^3
\end{bmatrix}$,\\ \; \\

$\overline{A}_W^3=\begin{bmatrix}
	4r^2w^2& 2r^6 + 3r^2w^2& 2r^6 + 3r^2w^2& 2r^4w + 3w^3\\
	2r^6 + 3r^2w^2& 2r^2w^2& 2r^2w^2& 2r^4w + 3w^3\\
	2r^6 + 3r^2w^2& 2r^2w^2& 2r^2w^2& 2r^4w + 3w^3\\
	2r^4w + 3w^3& 2r^4w + 3w^3& 2r^4w + 3w^3& 4r^2w^2
\end{bmatrix}$.\\ \; \\

Therefore $CWR_2(K) = \left(\text{trace}(\overline{A}_B\cdot A_B)/2,\;\;\text{trace}(\overline{A}_W\cdot A_W)/2\right) =  (4r + 3w, 2r^2 + 3w)$ and 
$CWR_3(K) = \left(\text{trace}(\overline{A}_B^3)/6,\;\;\text{trace}(\overline{A}_W^3)/6\right) = (2r^2w + w^3, 2r^2w^2).$

\end{example}

\section{Recursive relations}\label{s3}

For an integer $n>0$, let the alternating reduced non-split links $Btw_{2n+1}$,  $Wtr_{2n+1}$, $t_a$, $t_b$ and $t_c$ differ only locally by changing tangles as shown in Figure \ref{rys17_4}. We can derive the following recursive "skein" relations for $CWR_i(w, r)=(CB_i(w,r), CW_i(w,r))$, with the restrictions on diagrams.

\begin{figure}[h!t]
	\begin{center}
		\includegraphics[width=10cm]{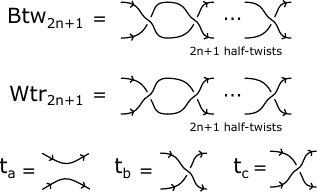}
		\caption{Definition of used tangles.\label{rys17_4}}
	\end{center}
\end{figure}

\begin{theorem}
	Assume that all diagrams used in the following relations are reduced, alternating, and non-split. Then, we have for all integers $k\geq 2, n \geq 0$:
	
\begin{enumerate}

	\item 	$CB_{k}(Wtr_{2n+1}) = r^{2n}CB_{k}(t_c)+(1-r^{2n})CB_{k}(t_a),$
	\item 	$CW_{k}(Btw_{2n+1}) = w^{2n}CW_{k}(t_b)+(1-w^{2n})CW_{k}(t_a).$
	
\end{enumerate}

\end{theorem}

\begin{proof}
	The proof goes by counting an invariant while taking care of the cases where the concerning cycles pass through given tangles shown in Figure \ref{rys17_4}.

	\par 
	Let us consider the first equation. The left-hand side $CB_{k}(Wtr_{2n+1})$ is calculated with all cycles of length $k$, the right-hand side is calculated considering cycles of length $k$ not passing through $Wtr_{2n+1}$ that is $CB_{k}(t_a)$ (from the assumption that $t_a$ and is reduced) and those cycles of length $k$ that passes through $Wtr_{2n+1}$ correspond one-to-one with the cycles that pass through $t_c$, because we have only one edge in $G_B^*$ passing through $Wtr_{2n+1}$ from the assumption that $Wtr_{2n+1}$ and is reduced. The weights of edges, that are considered in counting $(CB_{k}(t_c)-CB_{k}(t_a))$, change with the appropriate accumulated weigh scale by $r^{2n}$. Similar arguments are for the second equation, from the assumption that $Btw_{2n+1}$ and $t_b$ are reduced diagrams.
\end{proof}


\section{Values for knots in tables}\label{s4}
We computationally generate, the following Table\;\ref{table1} of the $CWR$ invariant of knots.

\noindent
\begin{footnotesize}
	\renewcommand{\arraystretch}{1.25}

		\begin{longtable}[ht]{r|r||l|}
						\caption{Knots and their $CWR$ invariant. \label{table1}}\\
			DT & Rolf	&  $CWR$ invariant\\
			\hline
			\endfirsthead
			\multicolumn{3}{c}
			{\tablename\ \thetable\ -- \textit{Continued from previous page}}\\
			DT & Rolf	&  $CWR$ invariant\\
			\hline
			\endhead
			\hline \multicolumn{3}{r}{\textit{Continued on next page}} \\
			\endfoot
			\hline
			\endlastfoot
			
			$K3a1$& $3_{1}$ & $((3w, w^3), (w^3, 0))$\\
			
			$K4a1$& $4_{1}$ & $((w^2 + 2r, r^2 + 2w), (r^2w^2, r^2w^2))$\\
			
			$K5a1$& $5_{2}$ & $ ((w^3 + 2w, w^2 + 3w), (w^5, 0), (0, w^5))$\\
			
			$K5a2$& $5_{1}$ & $((5w, w^5), (0,0), (0,0), (w^5, 0))$\\
			
			$K6a1$& $6_{3}$ & $((r^2 + r + 3w, w^2 + 3r + w), (r^3w + w^3, rw^3 + r^3), (r^3w^2, r^2w^3))$\\

			$K6a2$& $6_{2}$ & $((2r + 4w, w^3 + r^2 + w), (r^2w, r^2w^4), (w^4, 0), (r^2w^3, 0))$\\
			
			$K6a3$& $6_{1}$ & $((w^4 + 2r, r^2 + 4w), (r^2w^4, 0), (0,0), (0, r^2w^4))$\\
			
			$K7a1$& $7_{7}$ & $((4r + 3w, 2r^2 + 3w), (2r^2w + w^3, 2r^2w^2), (2r^2w^2, r^4w^2), (r^4w, 0))$\\
			
			$K7a2$& $7_{6}$ & $((w^2 + 2r + 3w, r^2 + w^2 + 3w), (w^3, r^2w^3 + r^2w^2), (r^2w^3, w^5), (r^2w^4, 0))$\\
			
			$K7a3$& $7_{5}$ & $((w^2 + 5w, w^3 + w^2 + 2w), (w^4, 0), (w^5, w^7), (w^5, 0))$\\
			
			$K7a4$& $7_{2}$ & $((w^5 + 2w, w^2 + 5w), (w^7, 0), (0,0), (0,0), (0, w^7))$\\
			
			$K7a5$& $7_{3}$ & $((w^3 + 4w, w^4 + 3w), (0,0), (0, w^7), (w^7, 0))$\\
			
			$K7a6$& $7_{4}$ & $((2w^3 + w, 7w), (w^7, 0), (0, 2w^4), (0,0), (0, w^6))$\\
			
			$K7a7$& $7_{1}$ & $((7w, w^7), (0,0), (0,0), (0,0), (0,0), (w^7, 0))$\\
	
			$K8a1$& $8_{14}$ & $((w^2 + 2r + 4w, r^2 + w^2 + 4w), (2w^4 + r^2w, r^2w^2), (r^2w^3 + w^4, w^5), (r^2w^3, r^2w^5))$\\
			
			$K8a2$& $8_{15}$ & $((w^2 + 6w, 2w^2 + 4w), (w^4 + 2w^3, 0), (2w^5, w^6 + 2w^5), (w^6, 0))$\\
			
			$K8a3$& $8_{10}$ & $((r^3 + r^2 + 3w, w^2 + 5r + w), (r^5w + w^3, 0), (r^5w^2, r^2w^3), (0, r^3w^3 + r^5))$\\
			
			$K8a4$& $8_{8}$ & $((r^2 + 3r + 3w, r^3 + w^2 + 2r + w), (w^3, r^3w^3 + r^5), (0, r^2w^3), (r^5w, 0), (r^5w^2, 0))$\\
			
			$K8a5$& $8_{12}$ & $((2w^2 + 4r, 2r^2 + 4w), (r^2w^2, r^2w^2), (r^2w^4, r^4w^2), (r^4w^2, r^2w^4))$\\
			
			$K8a6$& $8_{7}$ & $((r^4 + r + 3w, w^2 + 5r + w), (r^5w + w^3, rw^3), (r^5w^2, 0), (0, r^5), (0, r^4w^3))$\\
			
			$K8a7$& $8_{13}$ & $((5r + 3w, r^3 + w^2 + 2r + w), (r^2w + w^3, r^5 + rw^3), (r^4 + r^2w^2, r^4w^3), (r^4w, 0), (r^4w^2, 0))$\\
			
			$K8a8$& $8_{2}$ & $((2r + 6w, w^5 + r^2 + w), (r^2w, r^2w^6), (0,0), (0,0), (w^6, 0), (r^2w^5, 0))$\\
			
			$K8a9$& $8_{11}$ & $((w^3 + 2r + 3w, r^2 + w^2 + 4w), (r^2w, r^2w^3), (w^6, w^5), (r^2w^5, r^2w^4))$\\
			
			$K8a10$& $8_{6}$ & $((w^3 + 2r + 3w, w^3 + r^2 + 3w), (r^2w^3, 0), (w^6, 0), (r^2w^3, r^2w^6))$\\
			
			$K8a11$& $8_{1}$ & $((w^6 + 2r, r^2 + 6w), (r^2w^6, 0), (0,0), (0,0), (0,0), (0, r^2w^6))$\\
			
			$K8a12$& $8_{18}$ & $((4r + 4w, 4r + 4w), (4r^2w, 4rw^2), (4r^2w^2 + w^4, r^4 + 4r^2w^2), (4r^2w^3, 4r^3w^2))$\\
			
			$K8a13$& $8_{5}$ & $((2r + 6w, 2w^3 + r^2), (0, r^2w^6), (0,0), (2r^2w^3, 0), (w^6, 0))$\\
			
			$K8a14$& $8_{17}$ & $((r^2 + 2r + 4w, w^2 + 4r + 2w), (r^3w + r^2w, rw^3 + rw^2), (2r^3w^2 + w^4, 2r^2w^3 + r^4), $\\&&$(r^3w^3 + r^2w^3, r^3w^3 + r^3w^2))$\\
			
			$K8a15$& $8_{16}$ & $((2r^2 + r + 3w, 5r + 3w), (r^4w + 2r^3w + w^3, rw^2), (r^4w^2 + 2r^3w^2, 2r^2w^2),$\\&&$ (0, r^5 + 2r^3w^2), (0, r^4w^2))$\\
			
			$K8a16$& $8_{9}$ & $((r^3 + r + 4w, w^3 + 4r + w), (r^4w, rw^4), (w^4, r^4), (r^4w^3, r^3w^4))$\\
			
			$K8a17$& $8_{4}$ & $((w^4 + r^3 + r, 4r + 4w), (r^4w^4, 0), (0, r^4), (0, rw^4), (0,0), (0, r^3w^4))$\\
			
			$K8a18$& $8_{3}$ & $((w^4 + 4r, r^4 + 4w), (0,0), (0,0), (r^4w^4, r^4w^4))$\\
			
		\end{longtable}
	
\end{footnotesize}

We use our code, written in SageMath \cite{SageMath}. In this paper, there are the values for prime knots up to the crossing number equal to $8$. The values of the $CWR$ invariant for knots in the knot tables up to $16$ crossings can be found in the \texttt{CWRknots} folder in \url{https://drive.google.com/drive/folders/1mdF8zHY9Avmy1GnY4co3neH7q7b2Vnso}.


\begin{thebibliography}{99}
	
	\bibitem{KAtlas} D.~Bar-Natan. \emph{The Mathematica Package KnotTheory}, {Available at \url{https://katlas.org/wiki/Main_Page} (27/05/2024)}
	
	\bibitem{ChaKof08} A. Champanerkar and I. S. Kofman, On mutation and Khovanov homology, \emph{Commun. Contemp. Math.} 10 (2008), 973--992.
	
	\bibitem{Che13} Z. Cheng and H. Z. Gao, Mutation on knots and Whitney’s $2$-isomorphism theorem, \emph{Acta Math. Sin. (Engl. Ser.)} 29 (2013), no. 6, 1219--1230.
	
	\bibitem{SnapPy} M. Culler, N.M. Dunfield, M. Goerner and J.R. Weeks, \emph{Snap{P}y, a computer program for studying the geometry and topology of $3$-manifolds}, {Available at \url{http://snappy.computop.org} (27/05/2024)}
	
	\bibitem{SageMath} Developers, The~Sage, \emph{{S}agemath, the {S}age {M}athematics {S}oftware {S}ystem ({V}ersion 10.1)}, (2023),\\ \url{https://www.sagemath.org}
	
	\bibitem{HOMFLY85} P. Freyd, D. Yetter, J. Hoste, W. B. R. Lickorish, K. Millett and A. Ocneanu, A new polynomial invariant of knots and links, \emph{Bull. AMS} 12 (1985), 239–-246.
	
	\bibitem{Jab24} M. Jabłonowski, A Polynomial Pair Invariant of Alternating Knots and Links, (2024), accepted, \emph{J. Knot Theory Ramifications} \url{https://doi.org/10.1142/S0218216524500482}
	
	\bibitem{Kau88} L.H. Kauffman, New invariants in the theory of knots, \emph{Amer. Math. Monthly} 95 (1988), 195--242.
	
	\bibitem{Kau89} L.H. Kauffman, A Tutte polynomial for signed graphs, \emph{Discrete Applied Mathematics} 25 (1989), 105--127.
	
	\bibitem{Kau90} L.H. Kauffman, An invariant of regular isotopy, \emph{Transactions of the AMS} 318 (1990), 417-–471.
	
	\bibitem{KnaKol19} U. Knauer and K. Knauer, \emph{Algebraic graph theory.  Morphisms, Monoids and Matrices}, Vol.41 De Gruyter Studies in Mathematics, Walter de Gruyter GmbH, (2019).
	
	\bibitem{Laf13} K. Lafferty, The three-variable bracket polynomial for reduced, alternating links,
	\emph{Rose-Hulman Undergraduate Mathematics Journal} 14 (2013), 98–-113.
	
	\bibitem{Lis47} J.B. Listing, Vorstudien zur Topologie, \emph{Gottinger Studien(Abtheilung 1)} 1 (1847), 811--875.
	
	\bibitem{LivMoo23} C. Livingston and A.H. Moore, KnotInfo: Table of Knot Invariants, \\\url{https://knotinfo.math.indiana.edu/}, (June 2024).
	
	\bibitem{MenThi91} W. Menasco and M.B. Thistlethwaite, The Tait flyping conjecture, \emph{Bull. Amer. Math. Soc.} (1991), 403--412.
	
	\bibitem{MenThi93} W. Menasco and M.B. Thistlethwaite, The classification of alternating links, \emph{Annals of Mathematics} (1993), 113--171.
	
	
	
	\bibitem{Men21} W. Menasco, Alternating Knots, In \emph{Encyclopedia knot theory}, CRC Press (2021) 167--178.
	
	\bibitem{Mur87} K. Murasugi, Jones polynomials and classical conjectures in knot theory, \emph{Topology} 26.2 (1987), 187--194.
	
	\bibitem{PrzTra88} J. Przytycki and P. Traczyk, Invariants of links of the Conway type, \emph{Kobe J. Math.} 4 (1988), 115–-139.
	
	\bibitem{Sto24} A. Stoimenow, \emph{Knot data tables}, {Available at \url{https://stoimenov.net/stoimeno/homepage/ptab/index.html} (27/05/2024)}
	
	\bibitem{Thi87} M.B. Thistlethwaite, A spanning tree expansion of the Jones polynomial, \emph{Topology} 26 (1987), 297–-309
	
	\bibitem{Thi88} M.B. Thistlethwaite, Kauffman’s polynomial and alternating links, \emph{Topology} 27 (1988), 311--318.
	
\end{thebibliography}
\end{document}